\documentclass[11pt,reqno,a4paper]{amsart}
\usepackage{color,enumerate,cancel,hyperref,mathtools,tikz-cd}
\newtheorem{theorem}{Theorem}
\newtheorem*{theorem*}{Theorem}

\newtheorem{proposition}[theorem]{Proposition}
\newtheorem{lemma}[theorem]{Lemma}
\newtheorem{corollary}[theorem]{Corollary}

\theoremstyle{definition}

\theoremstyle{remark}
\newtheorem{remark}[theorem]{Remark}

\newcommand{\Lip}{\operatorname{Lip}_0}
\newcommand{\Lipp}{\operatorname{Lip}}

\newcommand{\R}{\mathbb{R}}
\newcommand{\N}{\mathbb{N}}

\newcommand{\F}{\mathcal{F}}

\newcommand{\Cell}{\operatorname{c}}

\newcommand{\Dens}{\operatorname{dens}}

\newcommand{\Weight}{\operatorname{w}}

\definecolor{orange}{rgb}{1,0.5,0}
\begin{document}
\title[On the geometry Banach spaces of the form $\Lip(C(K))$]{On the geometry of Banach spaces of the form $\Lip(C(K))$}
\author[L. Candido]{Leandro Candido}
\address{Universidade Federal de S\~ao Paulo - UNIFESP. Instituto de Ci\^encia e Tecnologia. Departamento de Matem\'atica. S\~ao Jos\'e dos Campos - SP, Brasil}
\email{leandro.candido@unifesp.br}
\author[P. L. Kaufmann]{Pedro L. Kaufmann}
\address{Universidade Federal de S\~ao Paulo - UNIFESP. Instituto de Ci\^encia e Tecnologia. Departamento de Matem\'atica. S\~ao Jos\'e dos Campos - SP, Brasil}
\email{plkaufmann@unifesp.br}


\thanks{Both authors were suported by grant 2016/25574-8, São Paulo Research Foundation (FAPESP). 
P. L. Kaufmann was supported additionally by grant 2017/18623-5, FAPESP}


\begin{abstract}We investigate the problem of classifying the Banach spaces $\Lip(C(K))$ for Hausdorff compacta $K$. In particular, sufficient conditions are established for a space  $\Lip(C(K))$ to be isomorphic to  $\Lip(c_0(\varGamma))$ for some uncountable set $\varGamma$. 
\end{abstract}

\maketitle

\section{Introduction}

In the past couple of decades, Banach spaces of real-valued Lipschitz functions and especially their canonical preduals, most often called \emph{Lipschitz-free spaces}, have received a lot of attention from many authors. Lipschitz maps between metric spaces can be interpreted as linear operators between the corresponding Lipschitz-free spaces. These Banach spaces thus encode the Lipschitz structure of metric spaces in a natural way (we refer to Weaver's book \cite{Weaver} - where these spaces are called Arens-Eells spaces - for a comprehensive introduction), and studying its geometry presents some challenges. For instance, when $X$ and $Y$ are Lipschitz equivalent metric spaces, the corresponding Lipschitz-free spaces $\F(X)$ and $\F(Y)$ are isomorphic, thus also the spaces of Lipschitz functions $\Lip(X)$ and $\Lip(Y)$ are isomorphic. It is natural to investigate converse statements. 
The first example of non-Lipschitz equivalent \emph{Banach} spaces with isomorphic Lipschitz-free spaces  was given by Dutrieux and Ferenczi in \cite{DutFec}. There they show that, for each infinite compact metric space $K$, $\F(C(K))$ is isomorphic to $\F(c_0)$, with an uniform bound on the Banach-Mazur distance. 
 It is worth pointing out that there are still fundamental problems which remain open in this sense. For instance, it is unknown whether $\F(\R^n)$ and $\F(\R^m)$ are isomorphic or not, when $n$ and $m$ are distinct and greater or equal to $2$. The same can be said even about $\Lip(\R^2)$ and $\Lip(X)$, where $X$ is a separable Banach space other than $\R$ or $\R^2$. The fact $\F(\R)$ and $\F(\R^2)$ are non-isomorphic can be derived from a deep result by Naor and Schechtman \cite{NaorSchet}. 

In this work we present a development in this direction, establishing sufficient conditions for a space  $\Lip(C(K))$ to be isomorphic to  $\Lip(c_0(\varGamma))$ for some uncountable set $\varGamma$.
We do so by studying the structure of $\Lip(X)$ and Lipschitz-free spaces, while exploring new tools that were discovered after the publication of \cite{DutFec}. 

The remainder of this article is organized as follows. 
In Section \ref{sec:term} we establish notation that will be used. In Section \ref{sec:main} we present our main results, which are proved in Sections \ref{Sec-Aux} and \ref{Sec-Proofs}. 

\section{Terminology}\label{sec:term}

In this section we establish the notation adopted throughout the article. Given a metric space $X$ with a distinguished point $0$, we denote by $\Lip(X)$ the Banach space of real-valued Lipschitz functions on $X$ that vanish at $0$, endowed with the norm of the \emph{smallest Lipschitz constant}, i.e. \[\|f\|_{\Lipp}=\sup_{x\neq y\in X} \frac{|f(x)-f(y)|}{d(x,y)}.\] 
The Lipschitz-free space over $X$, denoted by $\F(X)$, is the canonical predual to $\Lip(X)$ given by the closed linear span of the evaluation functionals 
in $\Lip(X)^*$. 
We refer the reader to \cite{God}, \cite{GodKal} and \cite{Weaver} for more details about these spaces. 

For a Hausdorff compactum $K$, we denote by $C(K)$ the Banach space of all continuous functions $f:K\to \R$, equipped with the norm \[\|f\|_{\infty}=\sup_{k\in K}|f(k)|.\]
Given set $\varGamma$, $\ell_{\infty}(\varGamma)$ denotes the space of all bounded functions $(a_i)_i\in \R^{\varGamma}$, endowed with the norm 
$\|(a_i)_i\|=\sup_{i\in \varGamma}|a_i|$. We denote by $c_0(\varGamma)$ the subspace of $\ell_{\infty}(\varGamma)$ consisting of all functions $(a_i)_i\in \R^{\varGamma}$ such that, for each $\epsilon>0$ the set $\{i\in \varGamma: |a_i|\geq \epsilon\}$ is finite. In particular, we denote $C_0(\N)$ by $c_0$ and $\ell_\infty(\N)=\ell_{\infty}$ as usual.
For a complete survey on $C(K)$ spaces, we refer the reader to \cite{Se}. 

Given an arbitrary topological space $X$, the \emph{weight} of $X$, denoted by $\Weight(X)$, is the smallest cardinality of a base for $X$. The \emph{density character} $\Dens(X)$ is the smallest cardinality of a dense subset of $X$. If $X$ is a Banach space, $\Dens(X)$ denotes the density of the norm topology and if $X$ is a dual space, $w^*$-$\Dens(X)$ denotes the density of its $w^*$-topology. A \emph{cellular family} in $X$ is a collection of nonempty pairwise disjoint open subsets of $X$. The \emph{cellularity} of $X$, denoted by $\Cell(X)$, is the supremum of the cardinalities of every cellular family in $X$. 

Given two Banach spaces $X$ and $Y$ we will write $X\sim Y$ when $X$ and $Y$ are isomorphic, and we write $X\cong Y$ when they are isometrically isomorphic. 
When a Banach space $Y$ contains a subspace that is isomorphic to $X$, we write $X\hookrightarrow Y$. If that subspace is, moreover, complemented in $Y$, we write $X\stackrel{c}{\hookrightarrow}Y$. 

All other standard terminology from Banach space theory and set-theoretic topology we will adopted as in \cite{Biorthog}. 

\section{Objective and main results}\label{sec:main}



In this section we outline and motivate the problem we are addressing, and state the main results obtained in this direction, which are Theorem \ref{Mainresultofthepaper} and its Corollary \ref{cor1}, and Theorem \ref{mainellinfty}.  Our focus is on investigating the geometry of Banach spaces of the form $\F(X)$ and $\Lip(X)$, where $X$ is a nonseparable $C(K)$ space. 
This research is motivated by the following result, already mentioned in the introduction: 

\begin{theorem}[Dutrieux, Ferenczi \cite{DutFec}]\label{DF} For any infinite metric compact space $K$,   the spaces $\F(C(K))$ and $\F(c_0)$ are isomorphic.
\end{theorem}
This gives in particular examples of pairs of Banach spaces having isomorphic Lipschitz-free spaces, but which are not even uniformly homeomorphic (due to a result by Johnson, Lindenstrauss and Schechtman, \cite[Theorem 3.1]{JohnLindSchecht}).
We want to obtain a similar result for general compacta. A first observation in that direction is that, if $\F(C(K))\sim \F(c_0(\kappa))$, for some infinite cardinal $\kappa$, then $\kappa=\Weight(K)$. Indeed, it is well known and straightforward, on one hand, that $\Dens(C(K))=\Weight(K)$ and $\Dens(c_0(\kappa))=\kappa$, and on the other hand, that $\Dens(\F(M))=\Dens(M)$ for any metric space $M$. But a little more can be said:

\begin{proposition}\label{lipisoweightequal}
 For each infinite Hausdorff compactum $K$ and each infinite cardinal $\kappa$, 
 \[\Lip(C(K))\sim \Lip(c_0(\kappa))\Rightarrow \Weight(K)=\kappa.\]
\begin{proof}
From Goldstine's theorem, $w^*$-$\Dens(X^{**}) \leq \Dens(X)$ for any Banach space $X$. Then, for any infinite metric space $M$, $w^*$-$\Dens(\Lip(M)^{*})\leq \Dens(\F(M))=\Dens(M)$. On the other hand, in \cite[Proposition 3]{HajNov} Hájek and Novotn\'y established that $\Lip(M)$ contains a copy of $\ell_{\infty}(\kappa)$, where $\kappa=\Dens(M)$. With an application of \cite[Fact 4.10]{Biorthog} we deduce that $w^*$-$\Dens(\Lip(M)^{*})\geq \Dens(M)$, and consequently
\[w^*\text{-}\Dens(\Lip(M)^{*})=\Dens(M).\]
The result follows immediately. 
\end{proof}
\end{proposition}

In view of Theorem \ref{DF}, we may pose the following question to guide our investigation.
\smallskip

\begin{quote}\textbf{Question 1:} let $K$ be an infinite Hausdorff compactum such that $\Weight(K)=\kappa$. Under which conditions on $K$ and $\kappa$ can we guarantee that $\F(C(K))$ and  $\F(c_0(\kappa))$ isomorphic, or at least that $\Lip(C(K))$ and  $\Lip(c_0(\kappa))$ isomorphic?
\end{quote}

\smallskip

As a part of this work, we obtain partial answers to Question 1 in the dual level, see Corollary \ref{cor1} and Theorem \ref{mainellinfty} below. We hope that the pursuit of a complete answer will stimulate the exploration of different techniques connecting the fields of $C(K)$ spaces and Lipschitz-free spaces, and their dual counterparts, $\Lip(X)$ spaces.  \medskip


For technical reasons that will become clear in the next section, in what follows we focus our attention on Hausdorff compact spaces $K$ with the property that $C(K)\sim C(L)$ for some zero-dimensional Hausdorff compact space $L$, depending on $K$. Let us denote by $\mathcal{B}$ the class of all such spaces.  In the sixties, Pe\l czy\'nski posed the question of whether every Hausdorff compact space was a member of $\mathcal{B}$. In \cite{Kosz} Koszmider answered negatively this question by constructing a counter-example. Later, other examples of the same type were obtained in \cite{AvKosh}. As this historical note suggests, $\mathcal{B}$ comprehends a wide class of Hausdorff compacta one may encounter. Our main result is the following: 

\begin{theorem}\label{Mainresultofthepaper}Let $K$ be an element of $\mathcal{B}$ of weight $\Weight(K)=\kappa$, having a cellular family of cardinality $\gamma$. Then
\[\Lip(c_0(\gamma))\stackrel{c}{\hookrightarrow}\Lip(C(K))\stackrel{c}{\hookrightarrow} \Lip(c_0(\kappa)).\]
\end{theorem}

When moreover $\gamma=\kappa$, from  \cite[Theorem 3.1]{Kauf} we obtain that the space $\Lip(C(K))$ is isomorphic to $\left(\bigoplus_{n\in \N}\Lip(C(K))\right)_{\ell_\infty}$. An application of Pe\l czy\'nski's decomposition method yields:

\begin{corollary}\label{cor1} 
Let $K$ be an element of $\mathcal{B}$ of weight $\Weight(K)=\kappa$, having a cellular family of size $\kappa$. Then 
\[\Lip(C(K))\sim \Lip(c_0(\kappa)).\] 
\end{corollary}


If $K$ is an Eberlein compact in $\mathcal{B}$ (we recall that a compact Hausdorff space is called an Eberlein compact if it is homeomorphic to a weakly compact subset of a Banach space), it follows from \cite[Theorem 4.2 and Remark (a) following it]{BenEle} that its cellurarity coincides with its weight and it is attained: there is in $K$ a cellular family of size $\Weight(K)$. Then Corollary \ref{cor1} applies and we deduce 
\begin{equation}\label{isoiso}\Lip(C(K))\sim \Lip(c_0(\Weight(K))). 
\end{equation}
The same conclusion holds for compact ordinal spaces $K$, which by \cite[Theorem 4.50]{Fabian} are not Eberlein when $K$ is uncountable,. 
If $\beta\N$ denotes the Stone-\v{C}ech compactification of $\N$ and $\beta \N^*=\beta\N\setminus \N$, it is well known that $C(\beta\N^*)\cong \ell_\infty/c_0$, $\Weight(\beta \N^*)=\mathbf{c}$ and $\beta \N^*$ has a cellular family of cardinality of the continuum, $\mathbf{c}$. Therefore, $\Lip(\ell_\infty/c_0)\sim \Lip(C(\beta\N^*))\sim \Lip(c_0(\mathbf{c}))$.

The relation (\ref{isoiso}) may be also be valid for $C(K)$ spaces where $K$ admits no cellular family of cardinality $\Weight(K)$, as we will see next. According to \cite[Theorem 7.13]{BenLind}, if $K$ is a scattered space of finite Cantor-Bendixson height and $\Weight(K)=\kappa$, then the space $C(K)$ is Lipschitz equivalent to $c_0(\kappa)$. Consequently, $\F(C(K))\sim \F(c_0(\kappa))$. We deduce that (\ref{isoiso}) holds for all scattered compacta of finite height. There are, however, examples of such spaces failing to satisfy the hypotheses of Corollary \ref{cor1}. For example, the Stone space of a Boolean algebra generated by finite subsets of $\N$ and an uncountable almost disjoint collection of infinite subsets of $\N$, see \cite{Mrowka}. In the literature these spaces are known as $\varPsi$-space or Mr\'owka-Isbell spaces and constitute examples of scattered compacta of height $3$, which are separable but not metrizable, implying that $\Cell(K)<\Weight(K)$.

Recalling that $\ell_\infty\cong C(\beta\N)$ and $\beta\N$ is separable, non-scattered and non-metrizable, we conclude our study on this topic with the following result:

\begin{theorem}\label{mainellinfty}If $\mathbf{c}$ denotes the cardinality of the continuum, then
\[\Lip(\ell_\infty)\sim \Lip(c_0(\mathbf{c})).\] 
\end{theorem}

$\ell_\infty$ and $c_0(\mathbf{c})$ are readily seen to be nonisomorphic, since  $\ell_\infty^*$ is $w^*$-separable while $c_0(\mathbf{c})^*$ is not. Since $\beta \N$ is non-scattered, they are not even uniformly homeomorphic, see \cite[Theorem 6.3]{JohnLindSchet}. 
\bigskip



\section{Auxiliary Results}
\label{Sec-Aux}

In this section we present some results needed to prove Theorems \ref{Mainresultofthepaper} and \ref{mainellinfty}. The first is a formula that relates the Lipschitz-free space of certain unions of metric spaces to the Lipschitz-free space of each component of that union, provided that a certain orthogonality condition is satisfied. For convenience of the reader, we include its simple proof.

\begin{proposition}\label{prop:orth}\emph{(\cite[Proposition 5.1]{Kauf}, \cite[Proposition 3.9]{Weaver})}. Let $(X,d)$ be a metric space with a distinguished point $0$ and let $(X_j)_{j\in \varGamma}$ be a family of subsets of X satisfying the following conditions: 
\begin{enumerate}
\item $X=\bigcup_{j\in\Gamma} X_j$;
\item $X_i\cap X_j = \{0\}$ for $i\neq j$;
\item \emph{(orthogonality)} there exists $C\geq 1$ such that, for all $i\neq j,\, x\in X_i$ and $y\in X_j$, $d(x,0)+d(y,0)\leq C\,d(x,y)$. 
\end{enumerate}
Then
\[\F \left(X\right) \sim \left(\bigoplus_{j\in \varGamma}\F (X_j)\right)_{\ell_1}.\]
\label{propbasic2}
\end{proposition}

\begin{proof}Assuming that each $(X_j,d)$ is a metric space having $0$ as distinguished point, we consider the function $\Phi:\left(\bigoplus_{j\in\varGamma}\Lip(X_j)\right)_{\ell_\infty}\to \Lip(X)$ given by 
$\Phi((f_j))(x)=f_i(x)$ if $x\in X_i$. For each $x\in X_r$ and $y\in X_s$ with $r\neq s$ we have
\begin{align*}
|\Phi((f_j))(x)-\Phi((f_j))(y)| &=|f_r(x)-f_s(y)|\leq \|f_r\|d(x,0) +\|f_s\|d(y,0)\\
& \leq C\max\{\|f_r\|,\|f_s\|\}d(x,y)
\leq C\|(f_j)\|d(x,y). 
\end{align*}
We deduce that $\Phi$ is a well defined bounded linear operator, easily seen to be also surjective. Recalling that on bounded sets of any $\Lip(X)$ space the weak$^*$ topology coincides with the topology of pointwise convergence, it is readily seen that $\Phi$ is weak$^*$ continuous. Thus, $\Phi$ is the adjoint of an isomorphism from $\F(X)$ onto $\left(\bigoplus_{j\in\varGamma}\F(X_j)\right)_{\ell_1}$. \end{proof}

We will also need the following strengthening of  \cite[Lemma 2]{DutFec}.

\begin{proposition}\label{Prop1}
Let $(Y_j)_{j\in \varGamma}$ be a family of subspaces of a Banach space $X$ and  $(P_j)_{j\in \varGamma}$, $P_j:X\to Y_j$, be a family of projections such that
\begin{enumerate}
\item $P_j[X]=Y_j$ for all $j$;
\item $P_j\circ P_i(x)=0$ for all $x$ whenever $i\neq j$;
\item the formula $x\mapsto (P_j(x))_{j\in \varGamma}$ defines a bounded injective linear ope\-ra\-tor from $X$ to $\left(\bigoplus_{j\in \varGamma} Y_j\right)_{c_0}$.
\end{enumerate} 
Then, 
\[\left(\bigoplus_{j\in \varGamma}\mathcal{F}(Y_j)\right)_{\ell_1}\stackrel{c}{\hookrightarrow}\mathcal{F}(X).\]
\end{proposition}
\begin{proof}
From condition (3), for each $x\in X$ we have $(P_j(x))_{j\in \varGamma}\in \left(\bigoplus_{j\in \varGamma} Y_j\right)_{c_0}$. Moreover, $x=0$ if and only if $P_j(x)=0$ for all $j\in \varGamma$. Therefore, for each $x\in X\setminus\{0\}$ there is $\beta(x)\in \varGamma$ such that $\|P_{\beta(x)}\|=\max_{i\in \varGamma}\|P_i(x)\|>0$. Write $Y=\bigcup_{i\in \varGamma} Y_j$ and define $R:X\to Y$ by 
\begin{displaymath}
R(x)=\left\{
\begin{array}{ll}
0& \text{ if }x=0\\
\left(1-\max_{i\neq \beta(x)}\left(\cfrac{\|P_i(x)\|}{\|P_{\beta(x)}(x)\|}\right)\right)P_{\beta(x)}(x)&\text{ if }x\neq 0.
\end{array} \right.
\end{displaymath}
If for some $j\in \varGamma$ and $x\in X$ we have $j\neq \beta(x)$ and $\|P_j(x)\|=\|P_{\beta(x)}(x)\|$, then $\max_{i\neq \beta(x)}\|P_i(x)\|=\|P_{\beta(x)}(x)\|$ and $R(x)=0$. We deduce that $R$ is well defined. Furthermore, from condition (2), if $x\neq 0$ then $\beta(R(x))=\beta(x)$ and $\max_{i\neq \beta(R(x))}\|P_i(R(x))\|=0$. It follows that $R(R(x))=R(x)$ for all $x\in X$. 

Let us verify that $R$ is Lipschitz, thus a Lipschitz retraction. In effect, let $x$ and $y$ be arbitrary elements of $X$. If $x\neq 0$ and $y=0$, then
\begin{align*}
 \|R(x)-R(y)\|&=\left\|\left(1-\max_{i\neq \beta(x)}\left(\cfrac{\|P_i(x)\|}{\|P_{\beta(x)}(x)\|}\right)\right)P_{\beta(x)}(x)\right\|\\
 &=\left\|\max_{i\neq \beta(x)}\{P_{\beta(x)}(x)-P_i(x)\}\right\|\leq  2\sup_{j\in \varGamma}\|P_j\|\|x-y\|.
\end{align*}

If $x\neq 0$ and $y \neq 0$ let us fix $j=\beta(x)$ and $r=\beta(y)$. If $j\neq r$,
\begin{align*}
\|R(x)-R(y)\|&=\left\|\left(1-\max_{i\neq j}\left(\frac{\|P_i(x)\|}{\|P_j(x)\|}\right)\right)P_j(x)-\left(1-\max_{i\neq r}\left(\frac{\|P_i(y)\|}{\|P_r(y)\|}\right)\right)P_r(y)\right\|\\
&=\left\|\left(\|P_j(x)\|-\max_{i\neq j}\|P_i(x)\|\right)\frac{P_j(x)}{\|P_j(x)\|}+\left(\|P_r(y)\|-\max_{i\neq r}\|P_i(y)\|\right)\frac{P_r(y)}{\|P_r(y)\|}\right\|\\
&\leq \left|\|P_j(x)\|-\max_{i\neq j}\|P_i(x)\|\right|+\left|\|P_r(y)\|-\max_{i\neq r}\|P_i(y)\|\right|\\
&\leq \left(\|P_j(x)\|-\|P_j(y)\|\right)+\left(\|P_r(y)\|-\|P_r(x)\|)\right)\\
&\leq \left(\|P_j(x-y)\|+\|P_r(x-y)\|\right)\leq 2\sup_{j\in \varGamma}\|P_j\|\|x-y\|.
\end{align*}
If $j=r$, without loss of generality we may assume that $\max_{i\neq j}\|P_i(x)\|\geq \max_{i\neq j}\|P_i(y)\|$. Let $s\in \varGamma$ be such that $\|P_s(x)\|=\max_{i\neq j}\|P_i(x)\|$. Then $\max_{i\neq j}\|P_i(y)\|\geq \|P_s(y)\|$, and 
\[|\max_{i\neq j}\|P_i(x)\|-\max_{i\neq j}\|P_i(y)\||\leq |\|P_s(x)\|-\|P_s(y)\||\leq \|P_s(x)-P_s(y)\|.\]
Therefore,
\begin{align*}
&\|R(x)-R(y)\|=\left\|\left(1-\max_{i\neq j}(\frac{\|P_i(x)\|}{\|P_j(x)\|})\right)P_j(x)-\left(1-\max_{i\neq j}(\frac{\|P_i(y)\|}{\|P_j(y)\|})\right)P_j(y)\right\|\\
&\leq \left(1-\max_{i\neq j}(\frac{\|P_i(x)\|}{\|P_j(x)\|})\right)\|P_j(x)-P_j(y)\|+\left\|\left(\max_{i\neq j}(\frac{\|P_i(y)\|}{\|P_j(y)\|})-\max_{i\neq j}(\frac{\|P_i(x)\|}{\|P_j(x)\|})\right)P_j(y)\right\|\\
&\leq \|P_j(x)-P_j(y)\|\\
&+\left\|\frac{\max_{i\neq j}\|P_i(x)\|}{\|P_j(x)\|} - \frac{\max_{i\neq j}\|P_i(x)\|}{\|P_j(y)\|}+ \left(\frac{\max_{i\neq j}\|P_i(x)\|  -   \max_{i\neq j}\|P_i(y)\|}{\|P_j(y)\|}\right)\right\|\|P_j(y)\|\\
&\leq \|P_j(x)-P_j(y)\| + \left(\frac{\max_{i\neq j}\|P_i(x)\|}{\|P_j(x)\|}\right)|\|P_j(x)\|-\|P_j(y)\||+|\max_{i\neq j}\|P_i(x)\|-\max_{i\neq j}\|P_i(y)\||\\
&\leq 2\|P_j(x)-P_j(y)\|+\|P_s(x)-P_s(y)\|\leq 3\sup_{j\in \varGamma}\|P_j\| \|x-y\|.
\end{align*}
From condition (3) and the Banach-Steinhaus theorem we may fix $A=\sup_{j\in \varGamma}\|P_j\|<\infty$ and deduce that for all $x,y \in X$,
\[\|R(x)-R(y)\|\leq 3A\|x-y\|.\]

Since is a Lipschitz retraction, from \cite[Lemma 3]{DutFec} it follows that
\[\F(\bigcup_{i\in \varGamma} Y_j)\stackrel{c}{\hookrightarrow}\mathcal{F}(X).\]
From conditions (2) and (3), there is $C>0$ such that $\max\{\|x_i\|,\|x_j\|\}\leq C\|x_i-x_j\|$ whenever $x_i\in Y_i$, $x_j\in Y_j$, thus Proposition \ref{propbasic2} applies, which means that
\[\mathcal{F}(Y)\sim \left(\bigoplus_{j\in \varGamma}\mathcal{F}(Y_j)\right)_{\ell_1},\]
and we are done. \end{proof}


\begin{corollary}\label{applic1}
If $\varGamma$ is an infinite set then
\[\left(\bigoplus_{j\in \varGamma}\mathcal{F}(c_0)\right)_{\ell_1}\stackrel{c}{\hookrightarrow}\mathcal{F}\left(c_0(\varGamma)\right).\]
Consequently,
\[\left(\bigoplus_{j\in \varGamma}\Lip(c_0)\right)_{\ell_\infty}\stackrel{c}{\hookrightarrow}\Lip\left(c_0(\varGamma)\right).\]
\end{corollary}
\begin{proof}
Putting $X= \left(\bigoplus_{j\in \varGamma}c_0\right)_{c_0}$ and $Y_j=c_0$ for each $j\in \varGamma$, by Proposition \ref{Prop1} we have  that
\[\left(\bigoplus_{j\in \varGamma}\F(c_0)\right)_{\ell_1}\stackrel{c}{\hookrightarrow}\F((\bigoplus_{j\in \varGamma}c_0)_{c_0}).\]
Since $\varGamma$ is infinite, then $c_0(\varGamma)\cong \left(\bigoplus_{j\in \varGamma}c_0\right)_{c_0}$ whence follows the first relation.
The second relation follows from the first by duality. 
\end{proof}

\begin{remark}From Proposition \ref{Prop1} we may also obtain other relations on Lipschitz free spaces of infinite sums of Banach spaces, as follows. For each $p\in [1,\infty)$ and a family of Banach spaces $(Y_i)_{i\in \varGamma}$ the following relation holds:
\[\left(\bigoplus_{j\in \varGamma}\mathcal{F}(Y_j)\right)_{\ell_1}\stackrel{c}{\hookrightarrow}\mathcal{F}((\bigoplus_{j\in \varGamma}Y_j)_{\ell_p}).\]
We may deduce for example that for each infinite set $\varGamma$ and for each $p\in [1,\infty)$,
\[\left(\bigoplus_{j\in \varGamma}\F(\ell_p)\right)_{\ell_1}\stackrel{c}{\hookrightarrow}\F(\ell_p(\varGamma)).\]
\end{remark}


\begin{lemma}\label{Lemma2}Let $(Y_j)_{j\in \varGamma}$ be a family of subsets of a pointed metric space $(X,d,0)$, 
$(R_j)_{j\in \varGamma}$ be a family of Lipschitz maps, $R_j:X\to Y_j$, fixing the origin $0$, such that $\sup_{j\in \varGamma}\|R_j\|_{\Lipp}=M<\infty$, and let $\mathcal{U}$ be an ultrafilter on $\varGamma$, such that $\displaystyle{\lim_{\mathcal{U}}R_j(x)=x}$ for all $x\in X$. Then,
\[\Lip(X)\stackrel{c}{\hookrightarrow} \left(\bigoplus_{j\in \varGamma}\Lip(Y_j)\right)_{\ell_\infty}.\]
\end{lemma}
\begin{proof}
The proof follows the same lines of \cite[Key Lemma]{CCD}. Consider an operator 
$T:\Lip(X)\to \left(\bigoplus_{j\in \varGamma}\Lip(Y_j)\right)_{\ell_{\infty}}$ given by
\[T(f)=(f|_{Y_j})_{j\in \varGamma},\]
where $f|_{Y_j}$ denotes the restriction to $Y_j$. It is evident that $T$ is a well defined linear operator with 
$\|T\|\leq 1$.

On the other hand, consider the map 
$S:\left(\bigoplus_{j\in \varGamma}\Lip(Y_j)\right)_{\ell_{\infty}}\to \Lip(X)$ defined by the formula
\[S((f_j)_{j\in \varGamma})(x)=\lim_{\mathcal{U}} f_j(R_j(x)).\]
Since, for each $x\in X$, $\sup_{j\in \varGamma}|f_j(R_j(x))|\leq M\|(f_j)_{j\in \varGamma}\|_{\ell_{\infty}}d(x,0)$, the previous 
limit always exists. Moreover
\small
\begin{align*}
|S((f_j)_{j\in \varGamma})(x) -  S((f_j)_{j \in \varGamma})(y)|&=|\lim_{\mathcal{U}} f_j(R_j(x))-\lim_{\mathcal{U}} f_j(R_j(y))|\\
&=|\lim_{\mathcal{U}}(f_j(R_j(x))-f_j(R_j(y)))|\\
&=\lim_{\mathcal{U}}\|(f_j)_{j\in \varGamma}\|_{\ell_{\infty}}d(R_j(x),R_j(y))\\
&\leq M\|(f_j)_{j\in \varGamma}\|_{\ell_{\infty}}d(x,y).
\end{align*}
\normalsize
We deduce that $S$ is bounded linear operator with $\|S\|\leq M$. Moreover, for every $f\in \Lip(X)$ and $x\in X$ we have
\[S(T(f))(x) = S((f|_{Y_j})_{j\in \varGamma}) = \lim_{\mathcal{U}}f|_{Y_j}(R_j(x))=\lim_{\mathcal{U}}f(R_j(x)) = f(x).\]
Therefore, $S\circ T=Id_{\Lip(X)}$ is the identity whence $P=T\circ S$ is a projection onto an isomorphic copy of $\Lip(X)$ in $\left(\bigoplus_{j\in \varGamma}\Lip(Y_j)\right)_{\ell_{\infty}}$. The conclusion follows.

\end{proof}

\begin{proposition}\label{zerodim}
If $K$ is an infinite Hausdorff zero-dimensional compactum of weight $\Weight(K)$, then
\[\Lip(C(K))\stackrel{c}{\hookrightarrow} \left(\bigoplus_{\Weight(K)}\Lip(c_0)\right)_{\ell_{\infty}}\]
\end{proposition}
\begin{proof}
Let $K$ be an infinite zero-dimensional compactum. Let $\varGamma$ be the set of all partitions of $K$ into a finite number 
of pairwise disjoint clopen sets from a fixed clopen basis with cardinality $\Weight(K)$. We consider on $\varGamma$ the following order: for every 
$i,j\in \varGamma$, $i\leq j$ if and only if $j$ is finer than $i$, that is, all elements of $i$ are union of elements of $j$. It follows that, endowed with this order, $\varGamma$ is a directed set.

For each nonempty clopen set $U$ we fix $x_U\in U$. We consider $Y_j$ as the collection of all functions that 
are constant in each element of the partition $j$. It is evident that $Y_j$ is isometric to $C(K_j)$ where $K_j$ is a finite 
set of the same cardinality as $j$. Let $R_j:C(K)\to C(K_j)$ the norm $1$ projection given by the formula
\[R_j(f)(x)=\sum_{U\in j}f(x_U)\chi_{U}.\]

Next, for each $f$ in $C(K)$ and for each $\epsilon>0$ there is a finite partition $j_0\in \varGamma$ such that
$|f(x)-f(y)|<\epsilon$ for each $x$, $y$ belonging to the same element $U\in j_0$. This implies that $\|f-R_{j_0}(f)\|<\epsilon$ and this is also true for $j\geq j_0$.

We deduce that for each $f\in C(K)$ the net $(R_j(f))_{j\in \varGamma}$ converges to $f$. Let $\mathcal{U}$ be an ultrafilter on $\varGamma$ containing all sets of the form $W_j=\{i\geq j:i \in \varGamma\}$, with $j \in \varGamma$. Then, $\lim_{\mathcal{U}}R_j(f)=f$ for every $f\in C(K)$. According to Lemma \ref{Lemma2}, 
\[\Lip(C(K))\stackrel{c}{\hookrightarrow} \left(\bigoplus_{j\in \varGamma}\Lip(C(K_j))\right)_{\ell_{\infty}}.\]
Since each $C(K_j)$ is 1-complemented in $c_0$ we deduce
\[\Lip(C(K))\stackrel{c}{\hookrightarrow} \left(\bigoplus_{j\in \varGamma}\Lip(c_0)\right)_{\ell_{\infty}}.\]
\end{proof}


Recall
that a metric space is called an absolute Lipschitz retract if there is a Lipschitz retraction from each of its superspaces onto it. To conclude this section, we point out that, as it happens for $c_0$, $c_0(\varGamma)$ is an absolute Lipschitz retract, for any ininite set $\varGamma$. The proof is an adaptation of  \cite[Example 1.5]{BenLind}:

\begin{proposition}\label{Retract}For any infinite subset $\varGamma$, $c_0(\varGamma)$ is an absolute Lipschitz retract.
\end{proposition}
\begin{proof}
For each $x\in \ell_{\infty}(\varGamma)$ let $d(x)=\inf\{\|x-a\|:a\in c_0(\varGamma)\},$
that is, $d(x)$ is the distance between $x$ and $c_0(\varGamma)$. 
We claim that for each $\epsilon>0$, the set $\varGamma_{\epsilon}(x)=\{i\in \varGamma: |x_i|\geq d(x)+\epsilon\}$ is finite. For otherwise, given $\epsilon>0$ let $a=(a_i)_i\in c_0(\varGamma)$ such that $d(x)\leq \|x-a\|<d(x)+\epsilon$.
If $\varGamma_{\epsilon}(x)$ is infinite, then for any $\delta>0$ there is $i\in \varGamma_{\epsilon}(x)$ such that $|a_i|<\delta$. Hence
\[d(x)+\epsilon\leq |x_i|\leq |x_i-a_i|+|a_i|\leq \|x-a\|+\delta.\]
We deduce that $d(x)+\epsilon\leq  \|x-a\|$, a contradiction.

We may then define a map $R:\ell_{\infty}(\varGamma)\to c_0(\varGamma)$ by the formula:
\begin{displaymath}
(R(x))_i=\left\{
\begin{array}{ll}
0& \text{ if }|x_i|< d(x)\\
(|x_i|-d(x))\mathrm{sign}(x_i) &\text{ if }|x_i|\geq d(x).
\end{array} \right.
\end{displaymath}
It is easily seen that $R(a)=a$ for each $a\in c_0(\varGamma)$, and  that $R$ is  $2$-Lipschitz. The conclusion follows from \cite[Lemma 1.1 and Proposition 1.2]{BenLind}.
\end{proof}


\section{Proof of main results}
\label{Sec-Proofs}

We are now ready to prove our main results.

\begin{proof}[Proof of Theorem \ref{Mainresultofthepaper}]
Let $K$ be a member of the class $\mathcal{B}$. Assume that $K$ has weight $\kappa$ and admits a cellular family of cardinality $\gamma$. According to a result of H. Rosenthal, see \cite{Rosenthal}, 
$C(K)$ has a linear isometric copy of $c_0(\gamma)$. From Proposition \ref{Retract}, there is a Lipschitz retraction from $C(K)$ onto $c_0(\gamma)$. Then, from \cite[Lemma 3]{DutFec}, 
\[\Lip(c_0(\gamma))\stackrel{c}{\hookrightarrow}\Lip(C(K)).\]

On the other hand, since $K$ belongs to the class $\mathcal{B}$, there is a zero-dimensional Hausdorff compactum $L$ such that $C(K)\sim C(L)$. The compacum $L$ has the same weight as $K$. Then, from Proposition \ref{zerodim} and Corollary \ref{applic1} we have
\[\Lip(C(L))\stackrel{c}{\hookrightarrow}\left(\bigoplus_{\kappa}\Lip(c_0)\right)_{\ell_\infty}\stackrel{c}{\hookrightarrow} \Lip(c_0(\kappa))\]
and we are done since $\Lip(C(K))\sim \Lip(C(L))$.
\end{proof}

\begin{proof}[Proof of Theorem \ref{mainellinfty}]
Fix a collection $\{N_j:j\in \mathbf{c}\}$ of infinite subsets of $\N$ such that $N_j\cap N_i$ is finite whenever $i\neq j$. Let $E$ be the closed linear span in $\ell_\infty$ of the set formed by the canonical copy of $c$ (the space of convergent sequences), together with the characteristic functions $\{\chi_{N_i}:i\in\mathbf{c}\}$. The space $E$ is is isometrically isomorphic to $C(K_E)$, where $K_E$ is an Isbell-Mr\'owka compactum as mentioned in the Section \ref{sec:main}, see \cite[Section 2.1]{KoszLau}. Since $K_E$ is a Hausdorff scattered compactum of finite Cantor-Bendixson height, $E$ is Lipschitz equivalent to $c_0(\mathbf{c})$  (\cite[Theorem 7.13]{BenLind}). Because $c_0(\mathbf{c})$ is an absolute Lipschitz retract (Proposition \ref{Retract}), it follows from \cite[Proposition 1.2]{BenLind} that $E$ is a Lipschitz retract of $\ell_{\infty}$. From \cite[Lemma 3]{DutFec},
\[\F(c_0(\mathbf{c}))\sim \F(E)\stackrel{c}{\hookrightarrow}\F(\ell_{\infty}),\]
and consequently,
\[\Lip(c_0(\mathbf{c}))\stackrel{c}{\hookrightarrow}\Lip(\ell_{\infty}).\]
On the other hand, since $\ell_\infty\cong C(\beta \N)$ and $\beta\N$ is a zero-dimensional compactum of weight $\mathbf{c}$, by Theorem \ref{Mainresultofthepaper} we have
\[\Lip(\ell_\infty)\cong \Lip(C(\beta\N))\stackrel{c}{\hookrightarrow} \Lip(c_0(\mathbf{c})).\]

Since $\Lip(\ell_\infty)$ is isomorphic to $\left(\bigoplus_{n\in \N}\Lip(\ell_\infty)\right)_{\ell_\infty}$ (\cite[Theorem 3.1]{Kauf}), an application of Pe\l czy\'nski's decomposition method yields the result. 
\end{proof}



\section{Acknowledgements}

We would like to thank Profs. Michal Doucha and Valentin Ferenczi for reading our manuscript and providing useful comments and Prof. Marek C\'uth for references. We are especially indebted to Prof. Piotr Koszmider for his patience in answering many questions about $C(K)$ spaces, and for the main idea behind the proof of Proposition \ref{zerodim}.

\bigskip


\bibliographystyle{plain}

\end{document}